\documentclass[11pt,a4paper]{article}
\usepackage{ucs}
\usepackage[utf8x]{inputenc}
\usepackage{amsmath,amssymb,,amscd, amsfonts,amsthm,latexsym} 
 \usepackage[all,cmtip]{xy}
 \usepackage{relsize}
\title{
       On the Constant Reductions of Automorphism Groups of valued Function Fields
       }
\author{
        Tovondrainy Christalin Razafindramahatsiaro \\
        African Institute of Mathematical Science\\
        6 Melrose road Muizenberg, 7945 Cape Town South Africa         
}
\date{\today}

\newcommand{\F}{\mathbb{F}}
\newcommand{\E}{\mathbb{E}}

\newcommand{\A}{\mathcal{A}}

\newcommand{\C}{\mathbb{C}}
\newcommand{\Aut}{\mathrm{Aut}}
\newtheorem{thm}{Theorem}[section]
\newtheorem{lem}[thm]{Lemma}

\newtheorem{pro}[thm]{Proposition}

\newtheorem{cor}[thm]{Corollary}
\newtheorem{rem}[thm]{Remark}

\begin{document}
\maketitle
\begin{abstract}

In this paper, we investigate properties of automorphism groups of function fields in one variable in relation to its reductions with respect to special valuations. In 1969, Deligne and Mumford proved that there exists a natural injective homomorphism between the automorphism groups of $\mathcal{X}_{\eta}$ and any special fibre of $\mathcal{X}.$ Here, we give a generalisation of this theorem in function field setting of Deuring's theory of constant reductions.

ACKNOWLEDGEMENT: I would to thank my advisor Barry Green for his constant guidance during my PhD study. I also want to express my gratitude to the African Institute of Mathematical Sciences (AIMS) for the financial support during the last three years.
\end{abstract}

\section{Introduction}
Let $X_g$ be a smooth projective irreducible curve of genus $g$ defined over an algebraically closed field $k.$ Denote by $p$ the characteristic of $k.$ Determining which groups can occur as automorphism groups $\mathrm{Aut}_k(X_g)$ of $X_g$ is a classic problem in mathematics. In the case when $g\leq 1,$ the problem is well understood. 

\

\textit{For convenience, in this note, we assume that the genus of a given curve is at least $2,$ unless otherwise specified.}

\

It is well known that $G=\mathrm{Aut}(X_g)$ is a finite group. So, one can ask: For a given $g,$ which finite groups can occur as automorphism groups on algebraic curves of genus $g$? And, conversely, for a given finite group $G,$ for which genera $g$ does there exist a curve of genus $g$ which has $G$ as automorphism group? 

In \cite{hu}, Hurwitz proved that the order of the group $G$ is less than or equal to $84(g-1)$ in characteristic $0.$ As an example, equality holds for the curve of genus $3$ defined by:
$$x^3y+y^3z+z^3x=0.$$ Such curves are called \texttt{Hurwitz curves.} Furthermore, in the case when $k$ is the field of complex numbers, there are  methods to find precisely which finite groups can occur on the curve $X_g.$ Indeed, first, recall that the category of algebraic curves over the complex numbers $\C$ is equivalent to the category of Riemann surfaces and fields of transcendence degree $1$ over $\C.$ So to determine which finite groups can act as groups of automorphisms of algebraic curves, it is sufficient to answer the inverse Galois problem for the rational function field $\C(x).$ That is what Hurwitz did. By Lefshetz Principle, Hurwitz results hold also over any algebraically closed field of characteristic $0.$ 

Unfortunately, the methods in the characteristic $0$ case do not seem to apply in positive characteristic. And there are very few results on this problem in positive characteristic. Although the group $G$ must be finite, there is no precise bound. In \cite{roq}, Roquette gave an example of a curve with automorphism group whose order is greater than $84(g-1).$ In \cite{hstich}, Stichtenoth proved that the order of $G$ must be less than $16g^4$ for $p>0.$ 

\

One way to understand why the methods in characteristic $0$ do not apply in general characteristic is the study of the reductions theory of curves. Indeed, we want to compare, for example, the automorphism group of a curve $X$ defined in characteristic $0$ with the automorphism group of the reduction $\overline{X}$ modulo a prime $p$ of $X.$ In \cite{delimum}, Deligne and Mumford proved that if $\mathcal{X}$ is a stable model of a smooth irreducible curve $X$ (defined in characteristic $0$ and the generic fiber of $\mathcal{X}$ is assumed to be isomorphic to $X$), then for any special fibre $\overline{X}$ of $\mathcal{X},$ there is a natural injective homomorphism from $\mathrm{Aut}(X)$ to $\mathrm{Aut}(\overline{X}).$ This solves partially the problem.

\

As we have mentioned above, this note is devoted to generalise this result of Deligne and Mumford. Our approach is similar to what Hurwitz did to solve the problem of determining finite groups than can occur as automorphism groups of algebraic curves in characteristic $0.$ We use the fact that the category of smooth projective irreducible curves over an algebraically closed field $k$ is equivalent to the category of function fields in one variable over $k.$ So, instead of working directly on curves, we will work on function fields in one variable. More precisely, we use Deuring's theory of constant reductions. 

Let $k$ be a field equipped with a valuation $v$\footnote{Note that the valuation $v$ is not necessary discrete nor a good reduction.} and let $kv$ be its residue field. One considers an algebraic function field $F{|}k$ of one variable and of genus $\geq2$ over $k.$ The valuation $v$ can be canonically extended to $F$ such that its reduction $Fv{|}kv$ is again an algebraic function field of one variable (Such prolongations are called constant prolongations). There is a well-defined canonical homomorphism (that we will define in section 3)
$$\phi_\mathcal{O}: Z(\mathcal{O})\rightarrow \mathrm{Aut}(Fw{|}kv)$$ where $\mathcal{O}$ denotes the valuation ring of $\F$ which corresponds to $v$ and $Z(\mathcal{O})$\footnote{$Z(\mathcal{O})=\mathrm{Aut}(\F{|}k)$ in the case when $v$ is a good reduction.}, the corresponding decomposition group.  Our main results in this context are the following: 

\

\textit{
Let $(\F{|}k,v)$ be a valued function field in one variable. Let $G{=}\mathrm{Aut}(\F{|}k).$ If $H$ is a subgroup of $G$ such that the extension $\F v{|}\F^{H}v$ is purely inseparable, then $H$ must be trivial.
}

Using this result (Lemma \ref{thm4}), we prove that:

\textit{
The homomorphism $\phi_{\mathcal{O}}$ defined above is injective. More precisely, we have $$Z(\mathcal{O})\simeq\Aut(\F v{|}\E v_\E)$$ where $\E$ is the fixed field of $G.$
}
\ 

\

 The later (Theorem \ref{thmT}) is a generalisation of the Deligne-Mumford theorem that we will present in the next section.

\

Note that the Deligne-Mumford theorem (hence its generalisation) is a very useful theorem in arithmetic geometry. Indeed, many researchers have used this theorem to solve important problems in mathematics. For instance, using a particular case of the theorem (in the case when the valuation is a good reduction), Kontogeorgis and Yang in \cite{Ariyang} was able to gave a list of the automorphism groups of hyperelliptic modular curves $X_0(N)$ in positive characteristic. In \cite{bi}, Green gave a partial answer to a question asked by Oort on Bounds on the number of automorphisms of curves over algebraically closed fields. We would like to point out that we have used Theorem \ref{thmT} with some results by Shaska \cite{shaska} to give the complete list of all automorphism groups of hyperelliptic curves in odd prime characteristic that can be lifted to characteristic $0.$  Also, Theorem \ref{thmT} can be used in the study of the Tchebotarev Density Theorem for function fields. These results will be published very soon.

\

\textit{Throughout this note, a function field is always a finite algebraic extension of a rational function field of transcendental degree one.}

\section{A Deligne-Mumford Theorem}

Let us consider a normal, connected, projective curve $X$ over $k.$ The arithmetic genera $p_a(X)$ of the curve $X$ coincides with the genus of the associated function field. Recall that the non-negative integer $g(X){:=}\mathrm{dim}_k\ H^1\left(X,\mathcal{O}_X\right)$ is the geometric genus of $X.$ The two invariants $p_a(X)$ and $g(X)$ are equal if $X$ is geometrically connected.

Assume that $k$ is an algebraic closure of a discretely valued field with prolongation $v$ to $k.$ Denote by $\mathcal{O}_k$ the valuation ring on $k$ with respect to $v.$ Note also that we may assume that the residue field which corresponds to the field $k$ is algebraically closed. Denote by $S$ the affine scheme $\mathrm{Spec}\left(\mathcal{O}_k\right)$ with generic point $\eta$ and closed point $s.$ 
\
A smooth projective curve over $\mathcal{O}_k$ does not always have stable reduction over $\mathcal{O}_k.$ An example is given by the projective curve defined over $\mathbb{Q}$ by:
$$x^4+y^4=z^4.$$ 
However, it is well known that
\begin{thm}\label{thmb3}
If $X$ is an irreducible smooth projective curve over $k$ of genus $g\geq2,$ then there exists a unique stable curve $\mathcal{X}$ over $\mathcal{O}_k$ such that the generic fibre $\mathcal{X}_\eta$ and $X$ are isomorphic. 
\end{thm}
For a reference, see \cite{b3}.

\

As far as we know, the following important theorem first appeared in \cite{delimum} by Deligne and Mumford. Besides, our aim is to generalize this theorem.
\begin{thm}\label{thm21}
Consider a stable model $$\mathcal{X}\rightarrow S$$ of genus $\geq2.$ Denote respectively by $\eta$ and $s$ the generic and closed points of $S$ and assume that the generic fibre $\mathcal{X}_\eta$ is smooth. Then, any automorphism $\sigma$ in $\mathrm{Aut}_k(\mathcal{X}_\eta)$ extends naturally to an automorphism in $\mathrm{Aut}_k(\mathcal{X}).$ Furthermore, the canonical homomorphism  $$\mathrm{Aut}_k(\mathcal{X}_\eta)\rightarrow\mathrm{Aut}_k(\mathcal{X}_s)$$ is injective.
\end{thm}
\begin{proof}
See \cite{delimum} Lemma \textbf{ I.12} and \cite{liu} Proposition 10.3.38.
\end{proof}

\

It is well known that the following two categories are equivalent (see \cite{robin} Corollary 6.12.):
\begin{itemize}
\item[(\textit{i})] Smooth projective curves over $k$, and dominant morphisms;
\item[(\textit{ii})]Function fields of one variable over $k$, and $k$-homomorphisms.
\end{itemize}
Therefore, Theorem \ref{thm21}, together with Theorem \ref{thmb3}, implies the following result:
\begin{cor}\label{cor21}
Let $(\F{|}k,v)$ be a valued function field where the valuation $v$ is a constant prolongation (prolongation which is a constant reduction) on $\F$ of the valuation on $k$ which corresponds to the discrete valuation ring $\mathcal{O}_{k}$ above. We assume that $v$ is a good reduction. Then, there exists a natural injective homomorphism  
\begin{equation}
\phi:\mathrm{Aut}(\F{|}k)\hookrightarrow\mathrm{Aut}(\F v{|}k v)
\end{equation}
 where $\F v{|}k v$ denotes the residue function field.
\end{cor}
Indeed, we can consider the smooth projective curve $X$ over $k$ associated to the function field $\F{|}k$ via the equivalent categories described above. Using Theorem \ref{thmb3}, there exists a stable curve $\mathcal{X}$ over $S$ which is a model of $X.$ The result follows immediately using Theorem \ref{thm21} and the fact that $\mathrm{Aut}_k(\mathcal{X}_\eta)=\mathrm{Aut}_k(X)=\mathrm{Aut}(\F{|}k)$ and $\mathrm{Aut}_k(\mathcal{X}_s)=\mathrm{Aut}(\F v{|}k v)$ where $\eta$ and $s$ are respectively the generic point and the closed point of the affine scheme $S.$

\

One natural question to ask is whether Corollary \ref{cor21} is still true for good non-discrete valuations. Following a suggestion and earlier work done by Roquette, Knaf has considered this question and got a positive answer in \cite{knaf}. Here, we generalise the result to the case where the valuation $v$ is not assumed to be good reduction nor discrete.

\section{A generalisation of the Deligne-Mumford Theorem}

Let $\F{|}k$ be a function field over the field of constants $k.$ Let us assume that the field $k$ is equipped with a valuation $v_{k}.$ Here, the valuation $v_k$ is not necessarily discrete. Let $v$ be a constant prolongation of $v_k$ to $\F.$  Denote by $p$ the characteristic of the residue field $k v.$ 

Denote respectively by $G$ and $\E$ the automorphism group $\Aut(\F{|}k)$ and the fixed field of $G$ in $\F.$ Since the group $G$ is finite as the genus $g\geq2,$ the set $V$ of prolongations of $v_\E{:=}v{|}_\E$ to $\F$ is finite of cardinal $t\geq1$. 

\

Note that we have,
$$\mathcal{O}'_{v_\E}=\displaystyle{\bigcap_{\mathcal{O}\in\A}}\mathcal{O} $$ where $\A$ denotes the set of the valuation rings of $\F$ which lie over the valuation ring $\mathcal{O}_v\cap\E$ of $\E.$

\

Let $\mathcal{O}\in{\A}$ and $$Z(\mathcal{O}):=\left\lbrace \sigma\in G|\sigma(\mathcal{O})=\mathcal{O}\right\rbrace $$ be \textit{the decomposition group} of $\mathcal{O}$ over $\E.$ The map
$$G\ni\sigma\mapsto\sigma \mathcal{O} \in \A$$ induces a bijection from $G/Z(\mathcal{O})$ into $\A.$ By definition, for any $\sigma\in G,$ we have $\sigma Z(\mathcal{O})\sigma^{-1}\subseteq Z(\sigma \mathcal{O})$ and $\sigma^{-1}Z(\sigma \mathcal{O})\sigma\subseteq Z(\mathcal{O}).$  So for any $\sigma\in G,$
\begin{equation}
Z(\sigma \mathcal{O})=\sigma Z(\mathcal{O}) \sigma^{-1}.
\end{equation}

The next proposition was inspired from \cite{endler}.
\begin{pro}\label{thmm}
Let $\pi$ be a place corresponding to $\mathcal{O}.$ There is a natural homomorphism 
\begin{equation}
\phi_{\pi}: Z(\mathcal{O})\rightarrow \mathrm{Aut}(\mathbb{F}v{\mid}\E v_\E)\hookrightarrow \mathrm{Aut}(\mathbb{\F}v{\mid}k v)
\end{equation}
defined for any $\sigma\in Z(\mathcal{O})$ by $$\pi\circ\sigma=\phi_{\pi}(\sigma)\circ \pi.$$ 
Moreover, if we denote by $T(\mathcal{O})$ its kernel called \textit{the inertia group} of $\mathcal{O}$ over $\E$, then we have:

\begin{enumerate}
\item[i.] ${T}(\mathcal{O})=\left\lbrace \sigma \in G \mid \sigma(x)-x\in \mathcal{M}_\mathcal{O}, \text{for all} \ x\in \mathcal{O}\right\rbrace;$
 \item[ii.] ${T}(\sigma(\mathcal{O}))=\sigma G^{T}(\mathcal{O}) \sigma^{-1}$ for all $\sigma$ in $G.$
 \item[iii.] ${T}(\mathcal{O})$ is a $p$-group and the extension $\F v{|}\F^{T(\mathcal{O})} v$ is purely inseparable where $\F^{T(\mathcal{O})}$ is the fixed field of ${T}(\mathcal{O})$ in $\F.$
 \end{enumerate}
\end{pro}
\begin{proof}

\begin{enumerate}
\item[i.] Let $\sigma\in Z(\mathcal{O})$ such that $\phi_\mathcal{O}(\sigma)=\mathrm{Id}_{\F v}.$ Then $[\phi_{\pi}(\sigma)](\pi(x))=\pi(\sigma(x))=\pi(x)$ for any $x\in \mathcal{O},$ therefore, $\sigma(x)-x\in \mathcal{M}_\mathcal{O}.$ Now, if for all $x$ in $\mathcal{O},$ we have $\sigma(x)-x\in \mathcal{M}_\mathcal{O}$ ($\sigma\in G$) which implies $$\left[ \pi(\sigma(x)-x)=0 \Leftrightarrow (\phi_\pi(\sigma))(\pi(x))=\pi(x)\right], $$ then $\phi_\pi(\sigma)=\mathrm{Id}_{\mathbb{F}v}, \sigma(\mathcal{O})\subseteq \mathcal{O}$ and $\sigma\in Z(\mathcal{O})$ since $\sigma(\mathcal{O})\in\A.$ This proof is the same as the one in \cite{endler} 19.1 c).
\item[ii.] By definition of $\pi,$ the corresponding place for $\sigma \mathcal{O}$ is just $\pi\circ\sigma^{-1}.$ So for any $\tau\in T(\mathcal{O})$ and $x\in\sigma \mathcal{O},$ we have:
$$\pi\circ\sigma^{-1}\left(\sigma\tau\sigma^{-1}(x)-x\right)=\pi\left(\tau\sigma^{-1}(x)-\sigma^{-1}(x)\right).$$ Since $x\in\sigma \mathcal{O},$ then $\sigma^{-1}(x)$ is in $\mathcal{O}.$ Using i. and the fact that $\tau$ is an element of $T(\mathcal{O}),$ we conclude that $\left(\tau\sigma^{-1}(x)-\sigma^{-1}(x)\right)$ is in $\mathcal{M}_\mathcal{O}.$ Therefore,  $\sigma\tau\sigma^{-1}(x)-x\in \sigma\mathcal{M}_\mathcal{O}.$ Hence,  $\sigma T(\mathcal{O})\sigma^{-1}\subseteq T(\sigma(\mathcal{O}))$ for any $\sigma$ in $G.$ Conversely, we have $\sigma^{-1} T(\sigma(\mathcal{O}))\sigma\subseteq T(\mathcal{O}).$
\item[iii-] The field $k$ is algebraically closed, so $v$ is unramified and the ramification group ${V}(\mathcal{O})$ coincides with the inertia group. Furthermore, $V(\mathcal{O})$ is the $p$-Sylow subgroup of $T(\mathcal{O})$ (\cite{endler} Theorem 20.18). Thus, $T(\mathcal{O})$ is a $p$-group and  $\F v{|}\F^{T(\mathcal{O})} v$ is purely inseparable.
\end{enumerate}
\end{proof}

\

Let us now define the $\mathcal{O}_{k}$-curve associate to the set of valuations $V$ as described in \cite{b3}. For that, let us make some convention of notations:

\underline{$\mathbf{Notations}$}:
\begin{itemize}
\item $\mathcal{R}_f:=(\mathcal{O}_{k}\left[f\right])'=R_f\cap \mathcal{O}_w;$
\item $R_f:=(k\left[f\right])'=\mathcal{R}_f\otimes_{\mathcal{O}_{k}}k;$
\item Since $f^{-1}$ is also $V$-regular, we define in the same way the rings $\mathcal{R}_{f^{-1}}$ and $R_{f^{-1}};$
\item $R_{fw}:=(k w\left[fw\right])'.$
\end{itemize}
 Recall that $w$ denotes the infnorm associated to $V$ (See \cite{b2}). The $\mathcal{O}_{k}$-curve, say $\mathcal{C}_f,$ is defined to be the $\mathcal{O}_{k}$-scheme 
 $$\mathcal{C}_f:=\mathrm{Spec}\mathcal{R}_f\cup\mathrm{Spec}\mathcal{R}_{f^{-1}}$$ obtained by glueing the affine $\mathcal{O}_{k}$-schemes $\mathrm{Spec}\mathcal{R}_{f}$ and $\mathrm{Spec}\mathcal{R}_{f^{-1}}$ along $\mathrm{Spec}(\mathcal{O}_{k}\left[f,f^{-1}\right])'.$
 
 \begin{thm}\label{thmg}
 The $\mathcal{O}_k$-curve $\mathcal{C}_f$ has the following properties:
 
 \begin{enumerate}
 
  \item The $\mathcal{O}_{k}$-curve depends only on $V$ in the following sense: 
  
  If $g$ is an another $V$-regular element for $\F{|}k,$ the corresponding $\mathcal{O}_k$-curve, $\mathcal{C}_g,$ is $\mathcal{O}_k$-isomorphic to $\mathcal{C}_f.$ We denote the curve by $\mathcal{C}_V.$
  \item The  $\mathcal{O}_k$-curve $\mathcal{C}_V$ is a projective integral normal flat $\mathcal{O}_k$-scheme of pure relative dimension $1.$ More precisely:
 $$\mathcal{C}_V\cong\mathrm{Proj}S$$ where 
 $$S=\bigoplus_{n\geq0}{\mathcal{L}(nD)\cap\mathcal{O}_w}$$ and $D$ is a pole divisor of a $V$-regular element for $\F{|}k.$ 
 \item The generic fibre of $\mathcal{C}_V$ is $k$-isomorphic to the non-singular irreducible projective curve $C$ associated to the function field $\F.$
 \end{enumerate}
 \end{thm}
 \begin{proof}
 \cite{b3} Theorem 1.1.
 \end{proof}
 
 \
 
Our first important result is the following lemma:

\begin{lem}\label{thm4}
Let $(\F{|}k,v)$ be a valued function field in one variable. Let $G{=}\mathrm{Aut}(\F{|}k).$ If $H$ is a subgroup of $G$ such that the extension $\F v{|}\F^{H}v$ is purely inseparable, then $H$ is trivial.
\end{lem}
\begin{proof}
Let $\sigma$ be an element of $H$ of order $>1.$  Denote by $\langle\sigma\rangle$ the subgroup of $H$ generated by $\sigma.$ Choose a regular transcendental element $f$ for the valuation $v.$ Let $\mathcal{C}_v$ be the $\mathcal{O}_k$-curve associated to $\left\lbrace v\right\rbrace .$ So the generic fibre of $\mathcal{C}_v$ is isomorphic to the curve
 $$C:=\mathrm{Spec}{R}_f\cup\mathrm{Spec}{R}_{f^{-1}}$$ which is the unique smooth projective curve with $\F$ as function field (Theorem \ref{thmg}). On the other hand, the closed fibre of $\mathcal{C}_v$ is isomorphic to the curve
 $$Cv:=\mathrm{Spec}\mathcal{R}_f v\cup\mathrm{Spec}\mathcal{R}_{f^{-1}}v=\mathcal{C}{\times} \left( k v\right) $$ which has $\F v$ as function field (but may have singularities in case $v$ is not a good reduction). Let us consider the smooth projective curve
 $$\overline{C}:=\mathrm{Spec}R_{fv}\cup\mathrm{Spec}R_{f^{-1}v}$$ which is the normalisation of $Cv.$
 
 Then we have:
 \begin{equation}
 g_{Cv}=g_{\overline{C}}+\delta,
 \end{equation}
where $\delta$ is the \textit{singularity number} (see \cite{liu} Propostion 7.5.4). And we have: 
$$\delta=\mathrm{dim}_{k v}\left(R_{fv}/\mathcal{R}_f v\right)+\mathrm{dim}_{k v}\left(R_{f^{-1}v}/\mathcal{R}_{f^{-1}} v\right).$$
Since the Euler-Poincar\'e characteristic does not change under reduction and $k$ is algebraically closed, we conclude that the curves $C$ and $Cv$ have the same arithmetic genus, i.e,
\begin{equation}
 g_{Cv}=g_{\overline{C}}+\delta=g_C.
 \end{equation}
 Furthermore, the extension $\F v{|}\F^{\langle\sigma\rangle} v$ is purely inseparable which implies 
 $$g_{\overline{C^r}}=g_{\overline{C}} \ (\cite{stich} \ \text{3.10})$$ where $C^r$ denotes the restriction of $C$ on $\F^{\langle\sigma\rangle}.$ The curve $\overline{C^r}$ is the normalization of the reduction of the curve $C^r.$ 
 
On the other hand, we have
 
 \begin{equation}
 g_{C^r}=g_{\overline{C^r}}+\delta^r,
 \end{equation}
 where
 $$\delta^r=\mathrm{dim}_{k v}\left(R_{fv}/\mathcal{R}^r_f v\right)+\mathrm{dim}_{k v}\left(R_{f^{-1}v}/\mathcal{R}^r_{f^{-1}} v\right)$$ and $\mathcal{R}^r_f v=R_{fv}\cap\E v$. The ring $\mathcal{R}^r_{f^{-1}} v$ is also defined in the same way as $\mathcal{R}^r_f v.$
 Since, $$\mathcal{R}^r_fv\subseteq\mathcal{R}_fv\subseteq R_{fv},$$ we conclude that $\delta^r\geq \delta.$ Therefore,
 $$g_{\F^{\langle\sigma\rangle}}=g_{C^r}=g_{\overline{C^r}}+\delta^r\geq g_{\overline{C}}+\delta=g_{\overline{C^r}}+\delta=g_C=g_{\F}.$$
 This can only happen if $g_{\F^{\langle\sigma\rangle}}=g_{\F}=0$ or $1$ by the Hurwitz genus formula. Thus, $g_{\F^{\langle\sigma\rangle}}=g_{\F}$ and $\F^{\langle\sigma\rangle}=\F.$ This contradicts the fact that the extension $\F{|}\F^{\langle\sigma\rangle}$ is Galois, hence, separable. Thus, $H$ is the trivial subgroup.
\end{proof}

\

Observe that:

\begin{rem}
In the proof of Lemma \ref{thm4}, we use the same technique as in the proof of Theorem \ref{thm21} in \cite{liu} (Proposition 10.3.38.). But here, we are not restricted to the case of stable reduction. We made the proof more general using directly the $\mathcal{O}_k$-curve associated to the valuation $v.$
\end{rem}

\

Now, we are able to prove our generalisation of the Deligne-Mumford theorem:

\begin{thm} \label{thmT}
The homomorphism $\phi_{\pi}$ defined above is injective. More precisely, we have $$Z(\mathcal{O})\simeq\Aut(\F v{|}\E v_\E)$$ where $\E$ is the fixed field of $G.$
\end{thm}
\begin{proof} According to a theorem in \cite{endler} (Theorem 19.6), the homomorphism $$\phi_{\pi}: Z(\mathcal{O})\rightarrow \mathrm{Aut}(\F v{|}\E v)$$ is surjective. However, by Proposition \ref{thmm} iii., the extension $\F v{|}\F^{T(\mathcal{O})}v$ is purely inseparable where $T(\mathcal{O})$ is the kernel of $\phi_{\pi}.$ Using Lemma \ref{thm4}, we conclude that $T(\mathcal{O})$ is a trivial group.
\end{proof}

\

\begin{rem}
In the previous theorem, let us assume that the valuation $v$ is a good reduction. So, the valuation $v$ is the only prolongation of the valuation $v_\E$ on $\E.$ Hence, we have:
$$Z(\mathcal{O})=\mathrm{Aut}(\F{|}k).$$ 

By Theorem \ref{thmT}, we conclude that
$$\mathrm{Aut}(\F{|}k)\subseteq\mathrm{Aut}(\F v{|}k v)$$ via the homomorphism $\phi_\pi.$ Hence, Theorem \ref{thmT} is a generalisation of the Knaf's theorem in \cite{knaf} which generalizes Theorem \ref{thm21} of Deligne and Mumford in the case when the reduction is good. 
\end{rem}

\

Moreover, we can generalize Theorem \ref{thmT} as follows:

\

Consider the infnorm $w$ with respect to the set of valuations $V.$ We have $$\F w:=\mathcal{O}'_{v_\E}/\left( \mathcal{M}_{\mathcal{O}_{v_\E}}\cdot \mathcal{O}_v'\right) =\displaystyle{\prod_{\mathcal{O}\in \A}} \mathcal{O}/\mathcal{M}_\mathcal{O}\simeq(\mathbb{F}v)^s.$$ Since the ring $\mathcal{O}'_{v_\E}$ is invariant under the action of $G,$ there exists an homomorphism 
$$\phi: G\ni \sigma\mapsto \phi(\sigma) \in \mathrm{Aut}(\mathbb{F}w{\mid}\E v_\E)\subseteq\mathrm{Aut}(\mathbb{F}w{\mid}k v)$$ such that for any $\sigma\in G,$
$$\psi\circ\sigma=\phi(\sigma)\circ\psi $$ where $\psi$ is the canonical homomorphism 
$$\psi: \mathcal{O}'_{v_\E}\rightarrow \F w$$ induced by the place $\pi$ corresponding to the valuation $\mathcal{O}$ above. Note that $\psi$ does not depend on which place we consider. Indeed, $\mathcal{O}'_{v_\E}$ is invariant under $G$ and any prolongation of $\mathcal{O}_{v_\E}$ in $\F$ is given by $\sigma(\mathcal{O})$ for some $\sigma\in G.$ The kernel of the homomorphism $\phi$ is 
$$\mathrm{Ker}\phi=\left\lbrace \sigma\in G\ {|} \ \sigma(x)-x\in \bigcap_{\mathcal{O}\in\mathcal{A}}\ \mathcal{M}_\mathcal{O}, \ \text{for all}\ x\in\mathcal{O}'_{v_\E}\right\rbrace.$$

We have:
\begin{pro}\label{lpro}
Consider the normal subgroup 
$$N{:=}\displaystyle{\bigcap_{\mathcal{O}\in\mathcal{A}}} \ Z(\mathcal{O})$$
of $G.$ Then, the restriction of the homomorphism $\phi$ on $N$ is injective. In particular, if $G$ is abelian, for a given valuation ring $\mathcal{O}$ in $\mathcal{A},$ we have
$$Z(O)\subseteq\mathrm{Aut}(\F w{|}\E v_\E)$$ via the homomorphism $\phi.$

\end{pro}
\begin{proof}
Let us denote by $T$ the kernel of the restriction of $\phi$ to the normal subgroup $N.$ 

For any $\sigma$ in $G,$ consider the following homomorphism which is defined by
\begin{align*}
h_\sigma \ : \ \F^{\times}/\E^{\times}&\rightarrow\F^\times\\
x\cdot\E^{\times}&\mapsto\frac{\sigma(x)}{x}.
\end{align*}
Denote respectively by $\Delta$ and $\Gamma$ the value group of $w$ and $v_\E.$ The mapping 
\begin{align*}
 \F&\rightarrow\Delta\\
x&\mapsto w(x)
\end{align*}
induces a sujective map $w^{\times}$ from $\F^{\times}/\E^{\times}$ to $\Delta/\Gamma.$ Note that for any $\sigma\in N$ and $x\in\F^\times$ we have
\begin{equation} \label{eqz}
h_\sigma(x\cdot\E^\times)\in U_{\mathcal{O}'_{v_\E}}.
\end{equation} 
Indeed, fix a valuation ring $\mathcal{O}$ in $\mathcal{A}.$ Denote by $v$ the corresponding valuation. For any $x\in\F^\times,$ we have 
$$v\left(\frac{\sigma(x)}{x}\right) =v\circ\sigma(x)-v(x).$$ But, since $\sigma$ is in $N,$ in particular, $\sigma$ belongs to $Z(\mathcal{O}).$ Hence, $v\circ\sigma=v.$ Thus, $v\left(\frac{\sigma(x)}{x}\right) \in U_{\mathcal{O}}.$ The valuation ring $\mathcal{O}$ being arbitrary, the statement \ref{eqz} holds. Furthermore, for any $\sigma\in T$ and $u\in U_{\mathcal{O}'_{v_\E}},$  we have:
$$ h_\sigma(u\cdot\E^\times)\in 1+\bigcap_{\mathcal{O}\in\mathcal{A}}\mathcal{M}_{\mathcal{O}}.$$ The set $U_{\mathcal{O}'_{v_\E}}$ denotes the group of all unit of the ring $\mathcal{O}'_{v_\E}.$ Since the vanishing set of the map $w^\times$ is the set
$$V_w=\left\lbrace u\cdot\E^\times \ {|} \ u\in U_{\mathcal{O}'_{v_\E}} \right\rbrace $$ and the kernel of the homomorphism $\psi\circ h_\sigma$ contains $V_w,$ we conclude, with the surjectivity of $w^\times,$ that for any $\sigma\in T,$ there exists a unique map $\overline{h}_\sigma\in \mathrm{Hom}(\Delta/\Gamma,\F w)$ such that 
$$\overline{h}_\sigma\circ w^\times=\psi\circ h_\sigma.$$ However, the base field $k$ is assumed to be algebraically closed. This implies that any prolongation of the valuation $v_\E$ on $\F$ is unramified. Hence, $\Delta=\Gamma.$ Let $\sigma\in T.$ Then, we have $\overline{h}_{\sigma}=\mathrm{Id}_{\Delta/\Gamma}.$ That is, for any $x\in\F^\times,$ we have 
$$\psi\left(\frac{\sigma(x)}{x} \right)=(\psi\circ h_{\sigma})(x\cdot\E^\times)=\mathrm{Id}_{\Delta/\Gamma}(w(x)+\Gamma)=1.$$
Therefore, for any $x\in\F^{\times},$ we conclude that
$$\frac{\sigma(x)}{x}-1\in \bigcap_{\mathcal{O}\in\mathcal{A}}\mathcal{M}_{\mathcal{O}}.$$
In particular, for all $x\in \mathcal{O}$ where $\mathcal{O}$ is any valuation ring in $\mathcal{A},$ we have
$$\sigma(x)-x\in \mathcal{M}_\mathcal{O}.$$ Thus, $\sigma\in T(\mathcal{O}),$ the kernel of the homomorphism $\phi_\pi$ defined above where $\pi$ is the place which corresponds to the valuation ring $\mathcal{O}.$ According to Theorem \ref{thmT}, $T(\mathcal{O})$ is the trivial group. Hence, $\sigma=\mathrm{Id}_{\F}.$ In fact, we have
$$T=\bigcap_{\mathcal{O}\in\mathcal{A}}T(\mathcal{O}).$$
\end{proof}

\

\begin{rem}
The proof of the previous proposition is, somehow, a generalisation of a theory developed in \cite{endler} to compute the ramification group $V(\mathcal{O})$ of a valuation ring $\mathcal{O}$ in $\mathcal{A}$ over $\E.$ Besides, recall that our proof of Theorem \ref{thmT} use the fact that the kernel $T(\mathcal{O})$ coincides with the ramification group $V(\mathcal{O})$ since $\Delta/\Gamma$ is trivial. Note that the group $V(\mathcal{O})$ is a $p$-subgroup of $G$ (see \cite{endler} Table p. 171).
\end{rem}


\begin{thebibliography}{1}

\bibitem{delimum} P. Deligne and D. Mumford, {\em The irreducibility of the space of curves of given genus}, Math. Inst. Hautes \'etudes Sci. 36 (1969), 75-109.
\bibitem{endler} O. Endler, {\em Valuation theory}, Springer Verlag, Berlin-Heidelberg-
New York, 1972.
\bibitem{b3} B. Green, M. Matignon, and F. Pop, {\em On valued function fields III,
Reduction of algebraic curves}, J. reine angew. Math. 432 (1992),
117–133.
\bibitem{bi} B. Green, {\em Bounds on the number of automorphisms of curves over
algebraically closed fields}, Israel journal of mathematics (2012), 1–8.
\bibitem{robin} R. Hartshorne, {\em Algebraic Geometry}, Graduate Texts in Mathemat-
ics, Springer Science+Business media, 2010.
\bibitem{hu} A. Hurwitz, {\em Uber algebraische Gebilde mit eindeutigen Transforma-
tionen in sich}, Math. Ann. 41 (1893), 403–442.
\bibitem{knaf} H. Knaf, {\em Reduktion von Funktionenkorpern und Automorphismen-
gruppen}, Diplomarbeit, Universitat Heidelberg (1990).
\bibitem{Ariyang} A. Kontogeorgis and Y. Yang, {\em Automorphisms of hyperelliptic mod-
ular curves x0 (n) in positive characteristic}, LMS J. Comput. Math.
13 (2000), 144–163.
\bibitem{liu} Q. Liu, {\em Algebraic Geometry and Arithmetic Curves, Oxford Gradu-
ate Texts in Mathematics}, Oxford University Press, 2002.
\bibitem{roq} P. Roquette, {\em Abschatzung der Automorphismenanzahl von Funkyio-
nenkorpern bei Primzahlcharakteristik}, math Z. 117 (1987), 157–
163.
\bibitem{shaska} T. Shaska, {\em Determining the automorphism group of hyperelliptic
curve}, Symbolic and algebraic computation: Proceedings of the 2003
international symposium (2003), 248–254.
\bibitem{hstich} H. Stichtenoth, {\em Uber die Automorphismengruppe eines alge-
braishen Funktionenkorpers von Primezahlcharakteristik. I. Eine
Abschatzung der Ordnung der Automorphismengruppe}, Arch. Math.
(Basel) 24 (1973), 527–544.
\bibitem{stich} H. Stichtenoth, {\em Algebraic function fields and codes}, Graduate Texts in Mathematics, Springer-Verlag, Berlin Heidelberg, 2009.

\end{thebibliography}
\end{document}